\documentclass{article}
\usepackage{amssymb}
\usepackage{amsmath}
\usepackage{amsthm}
\usepackage{mathabx}
\usepackage{enumerate}
\theoremstyle{plain}

\newtheorem{Cor}{Corollary}

\newtheorem{Thm}{Theorem}

\newcommand*{\Ann}{\ensuremath{\mathrm{Ann\,}}}
\newcommand*{\R}{\ensuremath{\mathbb{R}}}

\newcommand*{\Z}{\ensuremath{\mathbb{Z}}}
\newcommand*{\C}{\ensuremath{\mathbb{C}}}

\begin{document}
	
	\date{}
	
	\author{
		L\'aszl\'o Sz\'ekelyhidi\\
		{\small\it Institute of Mathematics, University of Debrecen,}\\
		{\small\rm e-mail: \tt szekely@science.unideb.hu,}
	}
	
	\title{Characterization of Locally Compact Abelian Groups Having Spectral Synthesis}

	\maketitle
	
	\begin{abstract}
		 In this paper we solve a long-standing problem which goes back to Laurent Schwartz’s work on mean periodic functions. Namely, we completely characterize those locally compact Abelian groups having spectral synthesis. So far a characterization theorem was available for discrete Abelian groups only. Here we use the localization concept for the ideals of the Fourier algebra of the underlying group. Recently we have shown that localizability of ideals is equivalent to synthesizability. Based on this equivalence we proved that if spectral synthesis holds on a locally compact Abelian group, then it holds on each extension of it by a a locally compact Abelian group consisting of compact elements, and also on any extension to a direct sum with a copy of the integers. Then, using Schwartz's result and Gurevich's counterexamples we apply the structure theory of locally compact Abelian groups to obtain our characterization theorem.  
	\end{abstract}

	\footnotetext[1]{The research was supported by the  the
	Hungarian National Foundation for Scientific Research (OTKA),
	Grant No.\ K-134191.}\footnotetext[2]{Keywords and phrases:
	variety, spectral synthesis}\footnotetext[3]{AMS (2000) Subject Classification: 43A45, 22D99}
	
	\section{Introduction}
	The study of spectral synthesis started with the fundamental paper of L.~Schwartz \cite{MR0023948}, where the following result was proved:
	
	\begin{Thm}\label{Sch}
		Every mean periodic function is the sum of a series of exponential monomials which are limits of linear combinations of translates of the function.
	\end{Thm}

Here "limit" is meant as uniform limit on compact sets. A continuous complex valued function on the reals is called {\it mean periodic} if the closure -- with respect to uniform convergence on compact sets -- of the linear span of its translates is a proper subspace in the space of all continuous complex valued functions. Calling this closure the variety of the function, the above result says that in the variety of each mean periodic function all exponential monomials span a dense subspace. If the function is not mean periodic, then its variety is the whole space of continuous complex valued functions, hence in this case even all polynomials span a dense subspace in it, by the Stone--Weierstrass theorem. 
\vskip.2cm

The basic concepts in this result can easily be generalized to more general situations. Here we shortly summarize the notation and terminology (see \cite{Sze23}). Given a commutative topological group $G$ we denote by $\mathcal C(G)$ the space of all continuous complex valued functions equipped with the topology of uniform convergence and with the pointwise addition and pointwise multiplication with scalars. If $f$ is in $\mathcal C(G)$ and $y$ is in $G$, then $\tau_yf$ denotes the {\it translate of $f$} defined by
$$
\tau_yf(x)=f(x+y)
$$
for each $x$ in $g$. A closed linear subspace $V$ in $\mathcal C(G)$ is called a {\it variety} on $G$ if it is {\it translation invariant}, that is, $\tau_yf$ is in $V$ for each $f$ in $V$ and $y$ in $G$.  Given an $f$ in $\mathcal C(G)$ the intersection of all varieties including $f$ is denoted by $\tau(f)$, and it is called the {\it variety of $f$}.
\vskip.2cm

Given a commutative topological group $G$ continuous complex homomorphisms of $G$ into the multiplicative group of nonzero complex numbers are called {\it exponentials}, and continuous complex homomorphisms of $G$ into the additive group of complex numbers are called {\it additive functions}. The elements of the function algebra in $\mathcal C(G)$ generated by all exponentials and additive functions are called {\it exponential polynomials}. Functions of the form 
\begin{equation}\label{poly}
f(x)=P\big(a_1(x),a_2(x),\dots,a_k(x)\big)m(x)
\end{equation}
are called {\it exponential monomials}, if $P:\C^k\to\C$ is a complex polynomial in $k$ variables, $a_1,a_2,\dots,a_k$ are additive functions, and $m$ is an exponential. Every exponential polynomial is a linear combination of exponential monomials. If $m=1$, then the above function is called a {\it polynomial}. 
\vskip.2cm

Exponential monomials can be characterized by difference operators. For each exponential $m$ and for every $y$ in $G$ we let $\Delta_{m;y}=\delta_{-y}-m(y)\delta_0$, $\delta_g$ being the point mass supported at the point $g$ in $G$.  We use the notation
$$
\Delta_{m;y_1,y_2,\dots,y_n}=\Delta_{m;y_1}*\Delta_{m;y_2}*\cdots *\Delta_{m;y_n}.
$$
In particular, we simply write $\Delta_{y_1,y_2,\dots,y_n}$ for $\Delta_{1;y_1,y_2,\dots,y_n}$.

\begin{Thm}\label{expmonchar}
	Let $G$ be a locally compact Abelian group and $m$ an exponential. The continuous function $f:G\to\C$ is an $m$-exponential monomial if and only if it is contained in a finite dimensional variety, and there is a natural number $n$ such that
	\begin{equation}\label{expmoneq}
		\Delta_{m;y_1,y_2,\dots,y_{n+1}}*f=0
	\end{equation}
	holds for each $y_1,y_2,\dots,y_{n+1}$ in $G$.
\end{Thm}

If we drop the condition of continuity of $f$ and $m$ in this theorem, then the solutions of \eqref{expmoneq} are called {\it generalized exponential $m$-monomials}. In particular, if $m=1$, then the solutions of \eqref{expmoneq} are called {\it generalized polynomials}.
The smallest natural number $n$  which  satisfies \eqref{expmoneq} is called the {\it degree} of $f$. 
\vskip.2cm

Using these concepts we say that the variety $V$ on $G$ is {\it synthesizable}, if  exponential monomials span a dense subspace in it. We say that {\it spectral synthesis} holds on $V$, if every subvariety of $V$ is synthesizable. We say that {\it spectral synthesis} holds on the group $G$, or $G$ is {\it synthesizable} if every variety on $G$ is synthesizable. Hence Schwartz's theorem can be formulated by saying that spectral synthesis holds on $\R$. In the paper \cite{MR0098951}, M.~Lefranc proved that spectral synthesis holds on $\Z^n$. In \cite{MR0177260}, R.~J.~Elliott made an attempt to prove that spectral synthesis holds on every discrete Abelian group, but his proof was incorrect. In fact, a counterexample for Elliott's statement was given in  \cite{MR2039084}. In \cite{MR2340978}, a characterization theorem was proved for discrete Abelian groups having spectral synthesis. Several papers studied spectral analysis and synthesis on different special locally compact Abelian groups ( see e.g. \cite{MR3137136, MR3438592}) but no general characterization theorem has been available so far.
\vskip.2cm

In the present paper we give a complete characterization of those locally compact Abelian groups on which spectral synthesis holds. Using the localization method we worked out in \cite{Sze23}, we show that if spectral synthesis holds on a locally compact Abelian group, then it also holds on its extensions by a  locally compact Abelian group consisting of compact elements (see \cite{Sze23a}), and on its extensions to a direct sum with the group of integers (see \cite{Sze23b}). Finally,  using the results of Schwartz \cite{MR0023948} and Gurevich \cite{MR0390759} we apply the structure theory of locally compact Abelian groups.

\section{Derivations and localization}

The basic result in \cite{Sze23} is a characterization of synthesizable varieties on a locally compact abelian group. To recall this result we need the concept of {\it localization of ideals}. Given a locally compact abelian group there is a one-to-one correspondence between all varieties  in the space $\mathcal C(G)$ of all continuous complex valued functions on $G$ and the closed ideals  in the measure algebra $\mathcal M_c(G)$ of all compactly supported complex Borel measures on $G$. This correspondence is given by the {\it annihilators}, and the basic property: for each variety $V$ in $\mathcal C(G)$ and closed ideal $I$ in $\mathcal M_c(G)$ we have
$$
V=\Ann \Ann V,\hskip2cm I=\Ann \Ann I.
$$
Here 
$$
\Ann V=\{\mu\in \mathcal M_c(G):\,\mu*f=0\enskip\text{for each}\enskip f\in V\}
$$
and
$$
\Ann I=\{f\in \mathcal C(G):\,\mu*f=0\enskip\text{for each}\enskip \mu \in I\}.
$$

We can reformulate this correspondence in terms of the ideals of the {\it Fourier algebra} $\mathcal A(G)$ of the group $G$. By the injectivity of the Fourier transform, the mapping $\mu\mapsto \widehat{\mu}$ sets up a topological algebra isomorphism between the measure algebra $\mathcal M_c(G)$ and the Fourier algebra $\mathcal A(G)$ of all Fourier transforms. For instance, any closed ideal in the Fourier algebra has the form $(\Ann V)\,\widehat{}$ with a unique variety $V$ in $\mathcal C(G)$. Accordingly, we say that a closed ideal $\widehat{I}$ in $\mathcal A(G)$ is synthesizable, if the variety $\Ann I$ in $\mathcal C(G)$ is synthesizable.
\vskip.2cm

We recall the concept of derivations on the Fourier algebra. The continuous linear operator $D:\mathcal A(G)\to \mathcal A(G)$ is called a {\it derivation of order one}, if 
$$
D(\widehat{\mu}\cdot \widehat{\nu})=D(\widehat{\mu})\cdot \widehat{\nu}+ \widehat{\mu}\cdot D(\widehat{\nu})
$$
holds for each $\widehat{\mu}, \widehat{\nu}$ in $\mathcal A(G)$. For each natural number $n\geq 1$, the continuous linear operator $D:\mathcal A(G)\to \mathcal A(G)$ is called a {\it derivation of order $n+1$}, if the bilinear operator
$$
(\widehat{\mu}, \widehat{\nu})\to D(\widehat{\mu}\cdot \widehat{\nu})-D(\widehat{\mu})\cdot \widehat{\nu}- \widehat{\mu}\cdot D(\widehat{\nu})
$$
is a derivation of order $n$ in both variables. All constant multiples of the identity operator on $\mathcal A(G)$ are considered derivations of order $0$. Finally, we call a linear operator on  $\mathcal A(G)$ a {\it derivation}, if it is a derivation of order $n$ for some natural number $n$. It is easy to see that all derivations on $\mathcal A(G)$ form an algebra with unit, which is the identity operator. The elements of the subalgebra generated by derivations of order not greater than one are called {\it polynomial derivations} -- in fact, they are polynomials of derivations of order at most one. Derivations on the Fourier algebra are completely characterized by the following result.

\begin{Thm}\label{gender}
	Let $G$ be a locally compact Abelian group. For each positive \hbox{integer $n$,} the continuous linear operator $D$ on the Fourier algebra of $G$ is a derivation of order at most $n\geq 1$ if and only if it has the form
	$$
	D\widehat{\mu}(m)=\int f_{D,m}\,\widecheck{m}\,d\mu,
	$$
	where $f_{D,m}$ is a generalized polynomial of degree at most $n$ vanishing at zero.
\end{Thm}

The generalized polynomial $f_{D,m}$ is uniquely determined by $D$ and $m$, and it is called the {\it generating function} of $D$. We note that $\widecheck{m}(x)=m(-x)$. In fact, it can be shown that $f_{D,m}$ is independent of $m$. The derivation $D$ is a polynomial derivation if and only if $f_{D,m}$ is a polynomial.

\begin{Thm}\label{localvar}
Let $G$ be a locally compact abelian group. The variety $V$ in $\mathcal C(G)$ is synthesizable if and only if it is localizable.
\end{Thm}

Given a derivation $D$ and an exponential $m$ we denote by $\widehat{I}_{D,m}$ the set of all functions $\widehat{\mu}$ in $\mathcal A(G)$ which are {\it annihilated at $m$} by all derivations of the form
$$
\widehat{\mu}\mapsto \int \varphi(x)\widecheck{m}(x)\,d\mu(x),
$$
where $\varphi$ belongs to the of $f_{D,m}$. It is easy to see that 
$$
I_{D,m}=\Ann \tau(\widecheck{f}_{D,m}m),
$$ 
which implies that $I_{D,m}$ is a closed ideal (see \cite{Sze23}).
As a by-product we obtain that so is the intersection $\widehat{I}_{\mathcal{D},m}=\bigcap_{D\in \mathcal D} \widehat{I}_{D,m}$, for any family $\mathcal D$ of derivations.
\vskip.2cm

Given a closed ideal $\widehat{I}$ in $\mathcal A(G)$ and an exponential $m$, the set of all derivations annihilating $\widehat{I}$ at $m$ is denoted by $\mathcal{D}_{\widehat{I},m}$. The subset of $\mathcal{D}_{\widehat{I},m}$ consisting of all polynomial derivations is denoted by $\mathcal{P}_{\widehat{I},m}$. Clearly, we have the inclusion
\begin{equation}\label{basic}
	\widehat{I}\subseteq \bigcap_m \widehat{I}_{\mathcal{D}_{\widehat{I},m},m}\subseteq \bigcap_m \widehat{I}_{\mathcal{P}_{\widehat{I},m},m}.
\end{equation}

The ideal $\widehat{I}$ is called {\it localizable}, if we have equalities in \eqref{basic}. The main result in \cite{Sze23} reads as follows:
\begin{Thm}\label{loc}
	Let $G$ be a locally compact Abelian group. The ideal $\widehat{I}$ in $\mathcal A(G)$ is synthesizable if and only if it is localizable.
\end{Thm}

\section{Characterization theorems}
In this section we completely characterize synthesizable locally compact Abelian groups using our localization results.. 
\vskip.2cm

First we recall the following result (see \cite[Theorem 1]{MR3589184}):

\begin{Thm}\label{compact}
	Let $G$ be a locally compact Abelian group and let $B$ denote the closed subgroup of $G$ consisting of all compact elements. Then spectral synthesis holds on $G$ if and only if it holds on $G/B$.
\end{Thm}

From this theorem we immediately derive the characterization of discrete synthesizable abelian groups, which was proved in the long and complicated paper \cite{MR2340978}.
	
\begin{Cor}\label{tors}
Spectral synthesis holds on a discrete Abelian group if and only if its torsion free rank is finite.
\end{Cor}

\begin{proof}
Indeed, if $G$ is discrete Abelian, then $B=T$, the set of all elements of finite order, hence $G/B=G/T$ is a homomorphic image of some power of $\Z$, which is the torsion free rank of $G$. If it is finite, then $G/T$ is synthesizable, as $\Z^k$ is synthesizable, by \cite{MR0098951}, and its homomorphic image is synthesizable, by the results in \cite{MR3589184}. On the other hand, if the torsion free rank is infinite, then by the results in \cite{MR2039084}, there are non-polynomial derivations on $G$, hence the chain of inclusions
$$
\widehat{I}\subseteq \widehat{I}_{\mathcal{D}_{\widehat{I},m},m}\subsetneqq \widehat{I}_{\mathcal{P}_{\widehat{I},m},m}\subseteq \bigcap_m \widehat{I}_{\mathcal{P}_{\widehat{I},m},m}
$$
shows that some ideal $\widehat{I}$ is non-localizable, hence spectral synthesis does not hold on $G$.
\end{proof}

We need another result for our characterization theorems \hbox{(see \cite[Lemma 4]{Sze24}):}

\begin{Thm}\label{zext}
	Let $G$ be a locally compact Abelian group. Then spectral synthesis holds on $G$ if and only if it holds on $G\times \Z$.
\end{Thm}

Now we are ready to settle the characterization problem of synthesizable locally compact abelian groups.

\begin{Cor}
	Let $G$ be a compactly generated locally compact Abelian group. Then spectral synthesis holds on $G$ if and only if $G$ is topologically isomorphic to $\R^a\times \Z^b\times F$, where $a\leq 1$ and $b$ are nonnegative integers, and $F$ is an arbitrary compact Abelian group.
\end{Cor}

\begin{proof}
	By the Structure Theorem of compactly generated locally compact Abelian groups (see \cite[(9.8) Theorem]{MR0156915}) $G$ is topologically isomorphic to
	$
	\R^a\times \Z^b\times F,
	$
	where $a,b$ are nonnegative integers, and $F$ is a compact Abelian group. If spectral synthesis holds on $G$, then it holds on its projection $\R^a$. By the results in \cite{MR0023948, MR0390759}, spectral synthesis holds on $\R^a$ if and only if $a\leq 1$, hence $G$ is topologically isomorphic to $\R^a\times \Z^b\times F$ where $a\leq 1$ and $b$ are nonnegative integers, and $F$ is a compact Abelian group.  
	\vskip.2cm
	
	Conversely, let $G=\R\times \Z^b\times F$ with $b$ is a nonnegative integer, and $F$ a compact Abelian group. By \cite{MR0023948}, spectral synthesis holds on $\R$. By repeated application of Theorem \ref{zext}, we have that spectral synthesis holds on $\R\times \Z^b$ with any nonnegative integer $b$. Finally, by Theorem \ref{compact}, spectral synthesis holds on $\R\times \Z^b\times F$. Our proof is complete.
\end{proof}

\begin{Cor}
	Let $G$ be a locally compact Abelian group.  Let $B$ denote the closed subgroup of all compact elements in $G$. Then spectral synthesis holds on $G$ if and only if $G/B$ is topologically isomorphic to $\R^n\times F$, where $n\leq 1$ is a nonnegative integer, and $F$ is a discrete torsion free Abelian group of finite rank.
\end{Cor}

\begin{proof}
	First we prove the necessity. If spectral synthesis holds on $G$, then it holds on $G/B$. By \cite[(24.34) Theorem]{MR0156915}, $G/B$ has sufficiently enough real characters. By \cite[(24.35) Corollary]{MR0156915}, $G/B$ is topologically isomorphic to $\R^n\times F$, where $n$ is a nonnegative integer, and $F$ is a discrete torsion-free Abelian group. As spectral synthesis holds on $\R^n\times F$, it holds on the continuous projections $\R^n$ and $F$. Then we have $n\leq 1$, and the torsion-free rank of $F$ is finite, by \cite{MR2340978}. 
	\vskip.2cm
	
	For the sufficiency, if $F$ is a torsion-free discrete Abelian group with finite rank, then it is the (continuous) homomorphic image of $\Z^k$ with some nonnegative integer $k$. By repeated application of Theorem \ref{zext},  we have that spectral synthesis holds on $\R\times \Z^k$, and then it holds on its continuous homomorphic image $\R\times F$. Finally, by Theorem \ref{compact}, we have that spectral synthesis \hbox{holds on $G$. }
\end{proof}

\section{Statements and Declarations}
Data sharing not applicable to this article as no datasets were generated or analysed during the current study. There are no financial or non-financial interests that are directly or indirectly related to the work submitted for publication.

\end{document}